%% file: Asymptotic_Enumeration_Many_Components_2018.tex
\documentclass[11pt,a4paper]{article}
\usepackage{lipsum}
\usepackage{amsfonts,amsthm,amsmath}
\usepackage{graphicx}
\usepackage{epstopdf}
\usepackage{algorithmic}
\usepackage{mathrsfs}
\usepackage{mathtools}
\usepackage{framed}
\usepackage{url}
\usepackage{fullpage,paralist,hyperref}
\usepackage{comment}
\usepackage{mathpazo}

\usepackage{xcolor}
\hypersetup{
    colorlinks,
    linkcolor={red!50!black},
    citecolor={blue!50!black},
    urlcolor={blue!80!black}
}

\theoremstyle{Definition}
\newtheorem{Definition}{Definition}[section]

\theoremstyle{plain}
\newtheorem{lemma}[Definition]{Lemma}
\newtheorem{theorem}[Definition]{Theorem}

\numberwithin{equation}{section}

\DeclarePairedDelimiter\floor{\lfloor}{\rfloor}

\newcommand\given[1][]{\:#1\vert\:}

\begin{document}

\title{\Large Asymptotic Enumeration of Graph Classes with Many Components}
\author{{\scshape Konstantinos Panagiotou}\thanks{E-Mail: kpanagio@math.lmu.de.} \\
\and
{\scshape Leon Ramzews}\thanks{E-Mail: ramzews@math.lmu.de.}}
\date{Department of Mathematics \\ Ludwigs-Maximilians-Universit\"at M\"unchen }

\maketitle

\begin{abstract} 
\small\baselineskip=9pt We consider graph classes $\cal G$ in which every graph has components in a class $\mathcal{C}$ of connected graphs. We provide a framework for the asymptotic study of $\lvert\mathcal{G}_{n,N}\rvert$, the number of graphs in $\mathcal{G}$ with $n$ vertices and $N:=\floor*{\lambda n}$ components, where $\lambda\in(0,1)$. Assuming that the number of graphs with $n$ vertices in $\mathcal{C}$ satisfies
\begin{align*}
\lvert \mathcal{C}_n\rvert\sim b n^{-(1+\alpha)}\rho^{-n}n!, \quad n\to \infty
\end{align*}
for some $b,\rho>0$ and $\alpha>1$ -- a property commonly encountered in graph enumeration -- we show that 
\begin{align*}
\lvert\mathcal{G}_{n,N}\rvert\sim c(\lambda) n^{f(\lambda)} (\log n)^{g(\lambda)} \rho^{-n}h(\lambda)^{N}\frac{n!}{N!}, \quad n\to \infty
\end{align*}
for explicitly given $c(\lambda),f(\lambda),g(\lambda)$ and $h(\lambda)$. These functions are piecewise continuous with a discontinuity at a critical value $\lambda^{*}$, which we also determine. The central idea in our approach is to sample objects of $\cal G$ randomly by so-called Boltzmann generators in order to translate enumerative problems to the analysis of iid random variables. By that we are able to exploit local limit theorems and large deviation results well-known from probability theory to prove our claims. The main results are formulated for generic combinatorial classes satisfying the SET-construction.
\end{abstract}

\section{Introduction}
	\input{introduction}

\section{Main Result}
	\label{sec:main}
	\input{mainresult}

\section{Preparations for the Proof}
	\label{sec:prelim}
	\input{preliminaries}

\section{Proof of Main Result}
	\label{sec:proof}
	\input{proof}

\section{Examples}
	\label{sec:exa}
	\input{examples}

\bibliographystyle{acm}
\bibliography{references}

\end{document}

%% file: introduction.tex
\newcommand{\cG}{\mathcal{G}}
\newcommand{\cF}{\mathcal{F}}
\newcommand{\cP}{\mathcal{P}}

Let $\cG$ be a class of (labeled) graphs closed under isomorphism, for example the class $\cF$ of forests or the class $\cP$ of planar graphs. Further, let $\cG_{n,k}$ be the class of graphs in $\cG$ that have vertex set $\{1, \dots, n\} =: [n]$ and $k\in [n]$ components. An important question that has a long and rich history is concerned with the enumeration -- exact or asymptotic -- of graphs in $\cG$, for example with the determination of $g_{n,k} := |\cG_{n,k}|$, see \cite{Bodirsky2007,Gimenez2005,Rue2013,panagiotou2016}. We also abbreviate $\cG_n = \cup_{k \ge 1} \cG_{n,k}$ and $g_n := |\cG_n|$.

A particular example is the case of so-called \emph{smooth and addable} classes of graphs. Following~\cite{McDiarmid2006}, a non-empty graph class $\cG$ is called \emph{weakly addable}, if for any $G \in \cG$ the graph obtained by connecting any two vertices in distinct components of $G$ is in $\cG$ as well. If in addition a graph $G$ is  in $\cG$ if and only if each component of $G$ is in $\cG$, then $\cG$ is called \emph{addable}. Finally, we call $\cG$ \emph{smooth} if
\[
	\lim_{n\to\infty} \frac{g_n}{ng_{n-1}}
\]
exists and is finite. Many important classes are smooth and addable, including (but by far not limited to) $\cF$ and $\cP$ mentioned before. For such classes, in~\cite{McDiarmid2006} it was shown that the number of components in a uniformly drawn random graph from $\cG_n$ converges in distribution to  1 + Po($\lambda$), where $\lambda$ is a constant depending on $\cG$ only. From this result, we immediately obtain for $k \in \mathbb{N}$
\begin{equation}
\label{eq:gnm,mbounded}
	g_{n,k} \sim e^{-\lambda}\frac{\lambda^{k-1}}{(k-1)!} g_n
\end{equation}
as $n\to\infty$. For example, in the case of trees it turns out that $\lambda = 1/2$ and $f_n \sim e^{1/2} n^{n-2}$, $n\rightarrow\infty$, see for example the classical work~\cite{Renyi1959}. We obtain for $k \in \mathbb{N}$ the explicit formula
\[
	|\cF_{n,k}| \sim \frac1{2^{k-1} (k-1)!} n^{n-2}
\]
as $n\to\infty$. Note that the domain of applicability of \eqref{eq:gnm,mbounded} is limited: we first have to fix the number of components $k$, and then let the number of vertices grow large. In particular, in the case where the number of components is large, for example proportional to the number of vertices, much less is known. An important exception is the case of forests, where the following quite detailed result exposing two phase transitions in the behavior is known, see e.g.~\cite{Flajolet2009,kolchin}. Unless a base is given explicitly, logarithms are always to base $\mathrm{e}$.
\begin{theorem}
\label{main:thm:forests}
Let $N:=\floor*{\lambda n}$, where $\lambda\in(0,1)$. Then
\begin{align*}
\frac{N!}{n!}\lvert\cF_{n,N}\rvert\sim
\begin{cases}
c_{-}(\lambda)n^{-3/2}\mathrm{e}^n2^{-N} ,&\lambda\in(0,\frac12) \\
c n^{-2/3}\mathrm{e}^{n}2^{-N} ,&\lambda=\frac12 \\
c_{+}(\lambda)n^{-1/2}{x_\lambda}^{-n}T(x_\lambda)^N ,&\lambda\in(\frac12,1)
\end{cases},
\end{align*}
with
\begin{align*}
c_{-}(\lambda)=\sqrt{\frac{2}{\pi}}\frac{\lambda}{(1-2\lambda)^{5/2}},\quad c={3^{-1/3}\Gamma(1/3)^{-1}}\quad\text{and}\quad c_{+}(\lambda)=\sqrt{\frac{\lambda(2\lambda-1)}{2\pi(1-\lambda)}}.
\end{align*}
Further,
\begin{align*}
x_\lambda = 2(1-\lambda)\mathrm{e}^{-2(1-\lambda)}\quad\text{and}\quad T(x_\lambda)=2\lambda (1-\lambda).
\end{align*}
\end{theorem}
Note that we correct several typos and computational errors in previous presentations of this result:\ for example, in the case $\lambda=1/2$ the constant factor given in \cite{Flajolet2009} is incorrect, see also \cite{janson2012}. Further, in the case $\lambda\in(0,{1}/{2})$ we obtain the characteristic exponent ${3}/{2}$ instead of ${1}/{2}$ and we correct the absence of the factor $n$ in the exponents. In addition, the normalizing factor is ${N!}/{n!}$ instead of ${1}/{n!}$.

The main contribution of this paper is to generalize Theorem~\ref{main:thm:forests} to a broader setting by applying probabilistic methods. Our primary domain of interest lies in the study of several families of graphs, like the class of planar graphs, but it is not limited to that: in what follows, we consider general combinatorial classes comprised of objects having finite size, and formulate our main results (Theorems~\ref{main:bigthm:1}, \ref{main:bigthm:2} and~\ref{main:bigthm:3}) that apply in the case where the counting sequence has a specific asymptotic property. In particular, depending on the subexponential growth of the counting sequence, we find out that a variety of asymptotic behaviors with different numbers of phase transitions  becomes apparent.

It should be mentioned that previous work has been done on this topic in the past; more particularly, theorems estimating large powers of generating functions were set up. Denote by $\mathcal{G}$ a class of (labeled) graphs and by $\mathcal{C}$ the subclass of connected graphs in $\mathcal{G}$. Let $C(x)$ be the exponential generating function of $\mathcal{C}$. Then the problem of finding the behavior of $g_{n,N}$ for large $n$ and $N:=\floor{\lambda n}$, $\lambda\in(0,1)$, can be reduced to deriving the coefficient of $x^n$ in $C(x)^{N}$, i.e. it can be easily verified that
\begin{align}
\label{introduction:eq:largepowerofgenfunc}
g_{n,N}=\frac{n!}{N!}[x^n]C(x)^{N}.
\end{align}
By imposing some (mild) hypotheses on $C(x)$, see for example \cite[Theorem VIII.8]{Flajolet2009} and references therein, the right-hand side in \eqref{introduction:eq:largepowerofgenfunc} can be evaluated asymptotically for $\lambda\in(\lambda^{*},1)$, where $\lambda^{*}$ is a critical value given in \eqref{main:eq:lambdastar}. Further, \cite[Theorem IX.16]{Flajolet2009} can be applied to the case $\lambda=\lambda^{*}$, whereas the assumptions on $C(x)$ given there are somewhat stricter than ours -- we only need the coefficients of $C(x)$ to be in a certain asymptotic regime. In general, our proofs are within the scope of a combinatorial and probabilistic setting as opposed to the analytic proofs in \cite{Flajolet2009}.

The paper is organized as follows. The general setting and our main results are presented in Section \ref{sec:main}. Auxiliary results such as properties of the \textit{Boltzmann model}, which plays a central role in our proofs, and limit theorems are presented in Section \ref{sec:prelim}. Indeed, our proof is similar in spirit to Kolchin's proof~\cite{kolchin} of Theorem~\ref{main:thm:forests}, where the counting problem is reduced to the problem of determining the probability that the sum of iid random variables equals a specific value. However, the setting considered here is more general, and we provide a systematic reduction through the aforementioned Boltzmann model. The proof of the main result is presented in Section \ref{sec:proof}, and examples are given in Section \ref{sec:exa}.


%% file: mainresult.tex
In order to formulate our main result we first need to introduce the notion of \textit{combinatorial species}, which include as specific examples classes of (labeled) graphs. We give only a concise introduction tailored to our specific application, and we refer to \cite{Leroux1998} for a detailed discussion and many examples, and to \cite{Flajolet2009} for the development of the equivalent language of \textit{combinatorial classes}.

A \textit{combinatorial species} is defined as a family of mappings $\mathcal{F}$ that maps any finite set $U$ (the ``labels'') to a finite set $\mathcal{F}[U]$ of $\mathcal{F}$-objects and any bijection $\sigma:U\rightarrow V$ to a bijective transport function $\mathcal{F}[\sigma]:\mathcal{F}[U]\rightarrow\mathcal{F}[V]$, such that the following properties are satisfied.
\begin{itemize}
\item For all bijections $\sigma:U\rightarrow V$, $\sigma ':V\rightarrow W$: $\mathcal{F}[\sigma '\circ\sigma]=\mathcal{F}[\sigma ']\circ\mathcal{F}[\sigma]$ and
\item let $\mathrm{id}_U:U\rightarrow U$ denote the identity map. Then $\mathcal{F}[\mathrm{id}_U]=\mathrm{id}_{\mathcal{F}[U]}$ for all finite sets $U$.
\end{itemize}
A basic example is the ``species of all graphs'': it maps each finite set $U$ to the set of all graphs with vertex set $U$, and each bijection $\sigma:U\rightarrow V$ naturally induces a bijection from the set of graphs with vertex set $U$ to the set of graphs with vertex set $V$.

We will need a bit more notation. Let $\mathcal{F}$ and $\mathcal{G}$ be species. We write $\mathcal{G}\subset\mathcal{F}$ and say that $\mathcal{G}$ is a \emph{subspecies} of $\mathcal{F}$, if $\mathcal{G}[U]\subset\mathcal{F}[U]$ for all finite $U$ and $\mathcal{G}[\sigma]=\mathcal{F}[\sigma]|_{\mathcal{G}[U]}$ for all bijections $\sigma:U\rightarrow V$. An example is the ``species of trees'' -- connected and acyclic graphs -- as a subspecies of the species of all graphs. 

Let $\mathcal{F}$ be a species. We say that $\gamma\in\mathcal{F}[U]$ has \emph{size} $\lvert\gamma\rvert:=\lvert U\rvert$ and $\gamma$ and $\gamma '$ are termed \emph{isomorphic} if there is a bijection $\sigma:U\rightarrow V$ such that $\mathcal{F}[\sigma](\gamma)=\gamma '$. For $n\in\mathbb{N}_0$ we let $\mathcal{F}_n:=\mathcal{F}[\{1,\dots,n\}]$ and by slight abuse of notation we will often identify $\mathcal{F}$ with $\bigcup_{n\in\mathbb{N}_0}\mathcal{F}_n$. The \textit{exponential generating series} of $\mathcal{F}$ is the formal power series 
\begin{align*}
F(x)=\sum\limits_{n\geq 0}\lvert\mathcal{F}_n\rvert\frac{x^n}{n!}.
\end{align*}
Note that $F$ may have radius of convergence zero. If this is not the case, we say that $\mathcal{F}$ is \textit{analytic}, and we call $F$ its \textit{exponential generating function} (egf). 
The framework of combinatorial species offers a whole bunch of \textit{constructions} that enable us to create new species from others, and which relate the corresponding generating series; these constructions appear frequently in modern theories of combinatorial analysis and in systematic approaches to random generation of combinatorial objects. For our needs, it suffices to present the \textit{set species} \text{{\scshape SET}} and, given two species $\mathcal{F}$ and $\mathcal{G}$, the \textit{substitution} $\mathcal{F}\circ\mathcal{G}$.
\begin{itemize}
\item $\text{{\scshape SET}}[U]=\{U\}$ for all finite sets $U$. The egf is given by
\begin{align*}
\sum\limits_{n\geq0}\frac{x^n}{n!}=\mathrm{e}^x.
\end{align*}
\item If the species $\mathcal{G}$ has no objects of size zero, then for all finite sets $U$
\begin{align*}
(\mathcal{F}\circ\mathcal{G})[U]=\bigcup\limits_{\pi\text{ partition of }U}\mathcal{F}[\pi]\times\prod\limits_{P\in\pi}\mathcal{G}[P].
\end{align*}
\end{itemize}

We can interpret an object in $\mathcal{F}\circ\mathcal{G}$ as an $\mathcal{F}$-object whose labels are substituted by objects from $\mathcal{G}$. The transport along a bijection $\sigma:U\rightarrow V$ is defined by applying the induced map $\sigma ':\pi\rightarrow\pi ':=\{\sigma(P):P\in\pi\}$ to the $\mathcal{F}$-object and the maps $\sigma|_P$, $P\in\pi$, to the (corresponding) objects in $\mathcal{G}$. Again, the notation for the substitution is suggestive: from the definition it also follows that the exponential generating series for $\mathcal{F}\circ\mathcal{G}$ equals $(F\circ G)(x)=F(G(x))$.

We are interested in the asymptotic enumeration of graph classes, where the graphs have ``many'' connected components. In terms of the theory of species this means that we consider a species $\mathcal{C}$ (of connected graphs) of which the components are taken from and the superordinate species is $\mathcal{G}:=\text{{\scshape SET}}\circ\mathcal{C}$. Note that the definition of \text{{\scshape SET}} and substitution species imply that any object in $\mathcal{G}_n$ can be equally seen as an unordered sequence of objects in $\mathcal{C}$ relabeled according to a partition $\pi$ of $[n]$. We define for $k \in \mathbb{N}_0$
\begin{align*}
\mathcal{G}_{n,k}=\bigcup\limits_{\pi\text{ is a $k$-partition of }[n]}\text{{\scshape SET}}[\pi]\times\prod\limits_{P\in\pi}\mathcal{C}[P]
\end{align*}
as the set of graphs having $k$ components in $\mathcal{C}$ and $n$ nodes. For an object $G$ in $\mathcal{G}_{n,k}$, we denote the number of components by $\kappa(G)=k$. In what follows, we derive the asymptotic behavior of $\lvert\mathcal{G}_{n,N}\rvert$ with $N:=\floor*{\lambda n}$ for $\lambda\in(0,1)$. 

Up to this point we have specified what classes of combinatorial objects we consider. We shall make an additional crucial assumption in the rest of this paper, namely that 
\begin{align}
\label{eq:heavyAssumption}
\mathcal{C}_1 \neq \emptyset \quad\text{and}\quad \lvert\mathcal{C}_n\rvert\sim bn^{-(1+\alpha)}\rho^{-n}n!,\quad n\rightarrow\infty, 
\end{align}
for some $b,\rho>0$ and $\alpha > 1$.
This asymptotic behavior is in general rather restrictive. However, for our intended applications it is not, as the asymptotic counting sequence of many relevant classes of graphs (trees, families of planar graphs, block-stable classes, \dots) has the above properties, see for example ~\cite{Rue2013} or \cite{Panagiotou2010}. In particular, for the case of so-called block-stable classes of graphs we always have $\alpha > 1$ and $\rho < 1$, as stated in~\cite{panagiotou2016}.

Before we proceed with presenting our results we need a technical preparation. The assumptions on $\lvert\mathcal{C}_n\rvert$ yield that $C(x)$ and $C'(x)$ are finite for all $x\in(0,\rho]$. Hence,
\begin{align}
\label{main:eq:lambdastar}
\lambda^{*}:=\frac{C(\rho)}{\rho C'(\rho)}
\end{align}
is well defined. The following simple lemma is also immediate, since the coefficients of $C(x)$ are non-negative.
\begin{lemma}\label{main:lem:uniquesolution}
For every $\lambda\in[\lambda^{*},1)$ there exists a unique $x_\lambda\in(0,\rho]$ such that
\begin{align}
\label{main:eq:uniquesolution}
\frac{x_\lambda C'(x_\lambda)}{C(x_\lambda)}=\frac{1}{\lambda}.
\end{align}
\end{lemma}
We also introduce some convenient notation. For $\lambda\in(0,\lambda^{*})$ define
\begin{align*}
c_{-}(\lambda):=\frac{b\lambda}{C(\rho) (1-\lambda/\lambda^{*})^{\alpha+1}}.
\end{align*}
In the case $\lambda=\lambda^{*}$ we will need
\begin{align*}
c:=\left(\frac{\alpha C(\rho)}{\lambda^{*}b\lvert\Gamma(1-\alpha)\rvert}\right)^{1/\alpha}\frac{1}{\lvert\Gamma(-1/\alpha)\rvert}.
\end{align*}
If $\lambda\in[\lambda^{*},1)$, denote by $x_\lambda$ the unique solution to \eqref{main:eq:uniquesolution}. Then set
\begin{align*}
c_{+}(\lambda):=(2\pi\sigma^2_\lambda\lambda)^{-1/2},
\end{align*}
where we abbreviate
\begin{align}
\label{main:eq:sigma}
\sigma^2_\lambda:=\frac{x_\lambda^2C''(x_\lambda)}{C(x_\lambda)}+\frac{1}{\lambda}-\frac{1}{\lambda^2}.
\end{align}
Finally, define
\begin{align*}
c_2:=\sqrt{\frac{C(\rho)}{b\pi\lambda^{*}}}.
\end{align*}
Our main theorems reflect the influence of the exponent $\alpha$ on the asymptotic behavior of $\lvert\mathcal{G}_{n,N}\rvert$. In particular, if $\alpha < 2$, then we have a similar behavior as in the case of forests (where $\alpha = 3/2$), where the critical exponent changes from $\alpha$ to $1/\alpha$ to $1/2$:
\begin{theorem}\label{main:bigthm:1}
Let $1<\alpha<2$. Then
\begin{align*}
\frac{N!}{n!}\lvert\mathcal{G}_{n,N}\rvert\sim
\begin{cases}
c_{-}(\lambda)n^{-\alpha}\rho^{-n}C(\rho)^N ,&\lambda \in (0,\lambda^{*}) \\
c n^{-1/\alpha}\rho^{-n}C(\rho)^N ,&\lambda=\lambda^{*} \\
c_{+}(\lambda)n^{-1/2}x_\lambda^{-n}C(x_\lambda)^N ,&\lambda \in (\lambda^{*},1)
\end{cases}.
\end{align*}
\end{theorem}
In the case $\alpha = 2$ the behavior is qualitatively similar, however, note the appearance of a logarithmic factor at the critical point:
\begin{theorem}\label{main:bigthm:2}
Let $\alpha=2$. Then
\begin{align*}
\frac{N!}{n!}\lvert\mathcal{G}_{n,N}\rvert\sim
\begin{cases}
c_{-}(\lambda)n^{-2}\rho^{-n}C(\rho)^{N} ,\mkern-18mu&\lambda \in (0,\lambda^{*}) \\
c_2(n\log(\lambda^{*}n))^{-1/2} \rho^{-n}C(\rho)^{N} ,\mkern-18mu&\lambda=\lambda^{*} \\
c_{+}(\lambda)n^{-1/2}x_\lambda^{-n} C(x_\lambda)^N ,\mkern-18mu&\lambda \in (\lambda^{*},1)
\end{cases}\!.
\end{align*}
\end{theorem}
Finally, in the remaining cases we have a different behavior with only two asymptotic regimes:
\begin{theorem}\label{main:bigthm:3}
Let $\alpha>2$. Then
\begin{align*}
\frac{N!}{n!}\lvert\mathcal{G}_{n,N}\rvert\sim
\begin{cases}
c_{-}(\lambda)n^{-\alpha}\rho^{-n} C(\rho)^N,&\lambda \in (0,\lambda^{*}) \\
c_{+}(\lambda)n^{-1/2}\rho^{-n}C(\rho)^N ,&\lambda \in [\lambda^{*},1)
\end{cases}.
\end{align*}
\end{theorem}

%% file: preliminaries.tex
In this section we collect some preliminary results that will be very handy in the forthcoming proofs.  We begin with the so-called \emph{Boltzmann model}, that will enable us to study the quantities of interest by reducing them to properties of independent random variables, see Lemma~\ref{prelim:lem:sumiid}.

\subsection*{Boltzmann Model} 
Recall that the egf  $F(x)$ of an analytic species $\cF$ is such that the radius of convergence is strictly greater than zero. Without explicitly mentioning it, we assume throughout that $x>0$ is chosen such that $F(x)$ is finite. The \textit{Boltzmann model}, introduced in \cite{duchon2004}, defines a random variable $\Gamma F(x)$ taking values in the entire species $\mathcal{F}$.
\begin{Definition}
The $\mathcal{F}$-valued random variable $\Gamma F(x)$ fulfilling
\begin{align*}
\mathbb{P}(\Gamma F(x)=F)=\frac{1}{F(x)}\frac{x^{\lvert F\rvert}}{\lvert F\rvert !}, \quad F\in\cup_{n \ge 0}\cF_n,
\end{align*}
is called \textit{Boltzmann generator} (with parameter $x$ for $\mathcal{F}$).
\end{Definition}
As mentioned in Section \ref{sec:main}, we identify $\mathcal{F}$ with $\bigcup_{n\in\mathbb{N}_0}\mathcal{F}_n$ for convenience and hence any object $F\in\mathcal{F}$ generated by $\Gamma F(x)$ is assumed to have labels in $[\lvert F\rvert ]$.
Simple computations show that 
\begin{equation}\label{prelim:eq:mean}
\mathbb{E}[\lvert\Gamma F(x)\rvert] =\frac{xF'(x)}{F(x)} \quad\text{and}\quad
\mathbb{E}[\lvert\Gamma F(x)\rvert^2] =\frac{x^2F''(x)+xF'(x)}{F(x)}.
\end{equation}
Note that the expressions above might not be finite for all $x$ within the radius of convergence of $F(x)$, compare to \eqref{proof:finitemean}--\eqref{proof:infinitevariance}.
An important property of the \textit{Boltzmann generator} is that for any $F$ in $\mathcal{F}_n$ the probability $\{\Gamma F(x)=F\}$ depends \emph{only} on $n$, which implies that
\begin{align*}
\mathbb{P}(\Gamma F(x)\in\mathcal{F}_n)=\frac{\lvert\mathcal{F}_n\rvert x^n}{F(x)n!}\quad\text{and}\quad
\mathbb{P}(\Gamma F(x)=F\given[\big]\Gamma F(x)\in\mathcal{F}_n)={\lvert\mathcal{F}_n\rvert}^{-1},
\end{align*}
i.e., the \textit{Boltzmann model} induces a uniform distribution on objects with the same size. \\

In the context of our intended application, recall that $\mathcal{G}=\text{{\scshape SET}}\circ\mathcal{C}$ is the species of which the components are objects of an underlying species $\mathcal{C}$. Then, an immediate consequence of the uniform distribution property is 
\begin{align}
\label{prelim:eq:gnasprob}
\mathbb{P}(\kappa(\Gamma G(x))=k\given[\big] \lvert\Gamma G(x)\rvert=n)=\frac{\lvert\mathcal{G}_{n,k}\rvert}{\lvert\mathcal{G}_n\rvert},k\in[n].
\end{align} 
Consequently, the problem of deriving $\lvert\mathcal{G}_{n,N}\rvert$ for $N:=\floor*{\lambda n}$ reduces to investigating the latter probability. Therefore, consider following algorithm, where each step is launched independently from the others.
\begin{enumerate}
\item Let $\kappa(x)$ follow the Poisson distribution with parameter $C(x)$.
\item Let $\gamma_1(x),\dots,\gamma_{\kappa(x)}(x)$ be iid copies of $\Gamma C(x)$.
\item Draw $L_1(x),\dots,L_{\kappa(x)}(x)$ uniformly at random from all  partitions of $\left[\sum_{i \leq \kappa(x)} \lvert\gamma_i(x)\rvert\right]$ with $\lvert L_i(x)\rvert = \lvert\gamma_i(x)\rvert$ for all $i\leq \kappa(x)$.
\item Relabel $\gamma_1(x),\dots,\gamma_{\kappa(x)}(x)$ canonically by $L_1(x),\dots,L_{\kappa(x)}(x)$.
\item Return an unordered sequence of the relabeled objects.
\end{enumerate}
Denote the outcome of this algorithm by $G^{*}(x)$. It is shown in \cite{duchon2004} that $G^{*}(x)$ indeed is a \textit{Boltzmann generator} for $\mathcal{G}$, that is, the following statement is true.
\begin{lemma}\label{prelim:lem:boltz}
For any $G\in \mathcal{G}$
\begin{align*}
\mathbb{P}(G^{*}(x)=G)=\mathbb{P}(\Gamma G(x)=G).
\end{align*}
\end{lemma}
This allows us to immediately translate our initial problem of determining $\lvert\mathcal{G}_{n,N}\rvert$ into the analysis of stochastic processes as shown in the next lemma.
\begin{lemma}\label{prelim:lem:sumiid}
Let $k\in [n]$ and $\gamma_1(x),\dots,\gamma_k(x)$ be iid random variables distributed like $\Gamma C(x)$. Then
\begin{align*}
\frac{k!}{n!}\lvert\mathcal{G}_{n,k}\rvert=C(x)^{k}x^{-n}\cdot\mathbb{P}\left(\sum\limits_{i\leq k}\lvert\gamma_i(x)\rvert=n\right).
\end{align*}
\end{lemma}
\begin{proof}
Our starting point is Equation \eqref{prelim:eq:gnasprob}. Note that $\kappa(\Gamma G(x))$ follows the $\mathrm{Poisson}$ distribution with parameter $C(x)$ due to Lemma \ref{prelim:lem:boltz}. Further, in the Boltzmann model the event $\{\lvert\Gamma G(x)\rvert=n\}$ occurs with probability $\lvert\mathcal{G}_n\rvert \mathrm{e}^{-C(x)}\frac{x^n}{n!}$; note that we used $G(x)=\mathrm{e}^{C(x)}$, which immediately follows by combining the egfs of the \text{{\scshape SET}} and substitution species. By  Bayes' theorem and Lemma \ref{prelim:lem:boltz} we obtain
\begin{align*}
\lvert\mathcal{G}_{n,k}\rvert&=\lvert\mathcal{G}_n\rvert\frac{\mathbb{P}(\lvert\Gamma G(x)\rvert =n\given[\big]\kappa(\Gamma G(x))=k)}{\mathbb{P}(\lvert\Gamma G(x))\rvert=n)} \cdot \mathbb{P}(\kappa(\Gamma G(x))=k) \\
&= \frac{n!}{k!}C(x)^{k}x^{-n}\mathbb{P}\left(\sum\limits_{i\leq k}\lvert\gamma_i(x)\rvert=n\right).
\end{align*}
\end{proof}
In the next subsections we will collect all necessary tools to study the asymptotic magnitude of the expression $\mathbb{P}(\sum_{i\leq N}\lvert\gamma_i(x)\rvert=n)$, that will eventually allow us to prove the main theorems.

\subsection*{Local Central Limit Theorem}
As before, let $\gamma_1(x),\dots,\gamma_{N}(x)$, where $N:=\floor*{\lambda n}$, denote iid copies of $\Gamma C(x)$.  Depending on the actual choice of $x$ the variance of $\lvert\Gamma C(x)\rvert$ is either finite or infinite, see Equations \eqref{proof:finitemean}--\eqref{proof:infinitevariance}, and the equation $\mathbb{E}[\sum_{i\leq N}\lvert\gamma_i(x)\rvert]=n+\mathcal{O}(1)$ allows for a unique solution, as stated in Lemma \ref{main:lem:uniquesolution}. We start with the case where both mean and variance of $\lvert\Gamma C(x)\rvert$ are finite.

We denote by $\mathcal{N}(0,\sigma^2)$ a normally distributed random variable with mean 0 and variance $\sigma^2>0$, i.e. the density is
\begin{align*}
\varphi_{\sigma^2}(x)=\frac{1}{\sqrt{2\pi\sigma^2}}\mathrm{e}^{-\frac{x^2}{2\sigma^2}}, \quad x\in\mathbb{R}.
\end{align*}
The following result is a local limit theorem for sums of independent random variables that have finite mean and variance, see \cite{Davis1995}.
\begin{lemma}
\label{prelim:lem:classlocal}
Let $(X_i)_{i\in\mathbb{N}}$ be independent, integer valued random variables and set $S_N:=\sum_{i\leq N}X_i$. Define $q_i:=\sum_{k\in\mathbb{Z}}\min\{\mathbb{P}(X_i=k),\mathbb{P}(X_i=k+1)\}$ and let $Q_N:=\sum_{i\leq N}q_i$. Assume there exist sequences $(a_i)_{i\in\mathbb{N}}$ and $(b_i)_{i\in\mathbb{N}}$ such that, as $N \to \infty$,
\begin{itemize}
\item $b_N>0$ for all $N$ and $b_N\rightarrow\infty$,
\item $b_N^{-1}(S_N-a_N) \overset{(\mathrm{d})}{\rightarrow}\mathcal{N}(0,1)$ and
\item $\sup_{k\leq N}{b_k^2}{Q_k^{-1}}$ has a finite limit.
\end{itemize}
Then, as $N \to \infty$, 
\begin{align*}
\sup\limits_{s\in\mathbb{Z}}\left\lvert b_N\mathbb{P}(S_N=s)-\varphi_1\left(\frac{s-a_N}{b_N}\right)\right\rvert\rightarrow 0.
\end{align*}
\end{lemma}
In our context, we will typically use $s=\mathbb{E}[S_N]+\mathcal{O}(1)$, where this theorem yields
\begin{align*}
\mathbb{P}(S_N=\mathbb{E}[S_N]+\mathcal{O}(1))\sim\frac{\varphi_1(0)}{b_N}=\frac{1}{b_N\sqrt{2\pi}}.
\end{align*}

\subsection*{Generalized Local Limit Theorem}
Next we consider the case where the variance of $\lvert\Gamma C(\rho)\rvert$ is infinite. In this case we will use a modified version of the local limit theorem, which is true for so-called stable random variables. 

The following notes are a summary of \cite[Chapter 2.2]{Embrechts2003}. Let $Y,X_1$ and $X_2$ denote real iid random variables. If for any non-negative $a,b$ there exist real $c>0$ and $d$ such that 
\begin{align*}
aX_1+bX_2\overset{(\mathrm{d})}{=}cY+d,
\end{align*}
then $Y$ (and its distribution function) is called \textit{stable}. Analytically, a stable random variable can be described by its characteristic function that is always of the form
\begin{align}\label{prelim:eq:stablecharac}
\Psi(t)=\mathrm{e}^{\mathrm{i}\gamma t-c\lvert t\rvert^{\alpha}(1-\mathrm{i}\beta\mathrm{sign}(t)z(t,\alpha))},\quad t\in\mathbb{R},
\end{align}
where $\gamma$ is real, $c>0$, $\alpha\in(0, 2]$, $\beta\in[-1, 1]$ and
\begin{align*}
z(t,\alpha)=
\begin{cases}
\tan\left(\sqrt{\frac{\pi\alpha}{2}}\right) ,&\alpha\neq 1 \\
-\frac{2}{\pi}\ln\lvert t\rvert ,&\alpha=1
\end{cases}.
\end{align*}
Let $(X_i)_{i\in\mathbb{N}}$ denote iid random variables with common distribution function $F$ and set $S_N:=\sum_{i\leq N}X_i$. If there exist sequences $(a_i)_{i\in\mathbb{N}}$ and $(b_i)_{i\in\mathbb{N}}$, where $b_N>0$ for all $N$, such that, as $N \to \infty$,
\begin{align}
\label{prelim:eq:stable1}
\frac{S_N-a_N}{b_N}\overset{(\mathrm{d})}{\rightarrow} Y
\end{align}
for some random variable $Y$, then $Y$ is necessarily in the class of stable random variables. If in addition the characteristic function of $Y$ is as in \eqref{prelim:eq:stablecharac} for $\alpha\in(0,2]$, then $Y$ is called $\alpha$-stable. We denote by $D_\alpha$ the $\alpha$-stable domain of attraction containing all distribution functions $F$ for which there exist sequences such that the limit in Equation \eqref{prelim:eq:stable1} converges to an $\alpha$-stable distribution. 

In the setup of Section \ref{sec:main}, the sum of iid copies of $\lvert\Gamma C(x)\rvert$ converges to an $\alpha$-stable random variable for $\alpha\in(0,2)$, cf.~\cite[Theorem 2.2.8]{Embrechts2003}. However, in our present setting it is in general not possible to compute the characteristic function of the limit; in particular, it is not possible to determine $\gamma,\beta$ and $c$. In this context, the following generalized local limit theorem for stable random variables derived from \cite[Theorem 4.2.1]{Ibragimov1971} will be useful.
\begin{lemma}
\label{prelim:lem:generallocallimit}
Let $(X_i)_{i\in\mathbb{N}}$ be iid copies of an integer valued random variable $X$ having distribution function $F$ in $D_\alpha$ for $\alpha\in(1, 2]$ and define $S_N:=\sum_{i\leq N}X_i$. Further, assume there exist sequences $(a_i)_{i\in\mathbb{N}}$ and $(b_i)_{i\in\mathbb{N}}$ such that, as $N\to\infty$,
\begin{itemize}
\item $b_N>0$ for all $N$ and
\item $b_N^{-1}(S_N-a_N) \overset{(\mathrm{d})}{\rightarrow} Y$, where $Y$ is $\alpha$-stable.
\end{itemize}
Let $X$ take values in the lattice $\{sh+c:s\in\mathbb{Z}\}$ for given integral constants $h,c$ and assume the span $h$ to be maximal. Then
\begin{align*}
\sup\limits_{s\in\mathbb{Z}}\left\lvert \frac{b_N}{h}\mathbb{P}(S_N=sh+cN)-g_\alpha\left(\frac{sh+cN-a_N}{b_N}\right)\right\rvert\rightarrow 0,
\end{align*}
where $g_\alpha$ is the density of $Y$.
\end{lemma}
In order to derive meaningful results from latter lemma, we will choose $nh+aN$ to be $\mathbb{E}[S_N]+\mathcal{O}(1)$ to obtain
\begin{align*}
\mathbb{P}(S_N=\mathbb{E}[S_N]+\mathcal{O}(1))\sim \frac{hg_\alpha(0)}{b_N}.
\end{align*}
The following results are extracted from \cite[Examples 5.5 and 5.10] {Janson2011}, where explicit calculations for $g_\alpha(0)$ are stated. 
\begin{lemma}
\label{prelim:lem:janson}
Let $\alpha\in (1, 2]$ and $X$ be a real-valued random variable. Suppose that there is $c>0$ such that, as $x \to \infty$,
\begin{align*}
\mathbb{P}(X>x)\sim cx^{-\alpha}\quad\text{and}\quad\mathbb{P}(X<-x)=o(x^{-\alpha}).
\end{align*}
Let $(X_i)_{i\in\mathbb{N}}$ be iid and distributed like $X$. Set $S_N:=\sum_{i\leq N}X_i$ and $a_N=\mathbb{E}[S_N]=N\mathbb{E}[X]$. If $\alpha\in(1,2)$, then as $N \to \infty$,
\begin{align*}
N^{-1/\alpha}(S_N-a_N)\overset{(\mathrm{d})}{\rightarrow} Y,
\end{align*}
where $Y$ is $\alpha$-stable with density evaluated at the origin
\begin{align*}
g_\alpha(0)=\frac{1}{(c\lvert\Gamma(1-\alpha)\rvert)^{1/\alpha}\lvert\Gamma(-1/\alpha)\rvert}.
\end{align*}
In the case $\alpha=2$ we obtain
\begin{align*}
(N\log N)^{-1/2}(S_N-a_N)\overset{(\mathrm{d})}{\rightarrow}\mathcal{N}(0,c).
\end{align*}
\end{lemma}

\subsection*{Large Deviations}
So far we have investigated the event $\{\sum_{i\leq N}\lvert\gamma_i(x)\rvert=n\}$ assuming that $\mathbb{E}[\sum_{i\leq N}\lvert\gamma_i(x)\rvert]=n+\mathcal{O}(1)$. Lemma \ref{main:lem:uniquesolution} does not ensure the existence of a pair $(\lambda, x)$ such that latter equation holds if $\lambda\in(0,\lambda^{*})$. In that case we will proceed as follows. First, recall that a positive measurable function $f$ is called slowly varying if
\begin{align*}
\frac{f(tx)}{f(t)}\rightarrow 1,\quad x\in\mathbb{R},
\end{align*}
as $t\rightarrow\infty$. In particular, functions asymptotic to a constant are slowly varying, hence we can write the probability $\mathbb{P}(\lvert\Gamma C(x)\rvert=n)$ as product of $n^{-(1+\alpha)}$ with a slowly varying function.

The next lemma is derived from \cite{Doney1989} and describes the local asymptotic behavior of sums of iid heavy-tailed random variables. 
\begin{lemma}
\label{prelim:lem:largedevi}
Let $(X_i)_{i\in\mathbb{N}}$ be iid integer valued random variables distributed like $X$ and set $S_N:=\sum_{i\leq N}X_i$. Assume that
\begin{itemize}
\item $\mathbb{E}[X]=\mu$ is finite and 
\item $\mathbb{P}(X=n)=n^{-(1+\alpha)}L(n)$ for some $1<\alpha<\infty$ and a slowly varying function $L$.
\end{itemize}
Then for any $\varepsilon>0$ and uniformly in $s\geq(\mu+\varepsilon)N$  
\begin{align*}
\mathbb{P}(S_N=s)\sim N\mathbb{P}(X=s-N\mu).
\end{align*}
\end{lemma}

%% file: proof.tex
We first prove the auxiliary Lemma \ref{main:lem:uniquesolution}.
\begin{proof}[Lemma \ref{main:lem:uniquesolution}]
$x C'(x)/C(x)$ is monotone increasing, as the coefficients of $C$ are non-negative. Further, as $\rho>0$ is the radius of convergence of $C(x)$, we have that $x C'(x)/C(x)\leq \rho C'(\rho)/C(\rho)$ for any $x$ such that $C(x)$ is finite. The assumption that $|\mathcal{C}_1|>0$ yields that $x C'(x)/C(x)\rightarrow 1$ as $x$ tends to $0$. The claim follows. 
\end{proof} \hfill$\square$

\begin{proof}[Theorems \ref{main:bigthm:1}, \ref{main:bigthm:2} and \ref{main:bigthm:3}] Let $\gamma(x)$ and $(\gamma_i(x))_{i\in\mathbb{N}}$ be iid copies of $\Gamma C(x)$.\ Due to Lemma \ref{prelim:lem:sumiid}, it suffices to compute the probability 
$$\mathbb{P}\Big(\sum_{i\leq N}\lvert\gamma_i(x)\rvert=n\Big), \quad \text{where}\quad N := \floor*{\lambda n},$$
that exposes a variety of asymptotic behaviors depending on the actual choice of $x$ and the parameters $\lambda$ and $\alpha$.
The assumptions on $\lvert\mathcal{C}_n\rvert$ yield that $C(x)$ and $xC'(x)$ are finite for all $x\in(0, \rho)$, as for any $\alpha>1$ 
\begin{align*}
x^n[x^n]C(x)=\lvert\mathcal{C}_n\rvert \frac{x^n}{n!}\sim b\left( \frac{x}{\rho}\right)^n n^{-(1+\alpha)}
\end{align*}
and
\begin{align*}
x^n[x^n]xC'(x)=n\lvert\mathcal{C}_n\rvert\frac{x^n}{n!}\sim b\left( \frac{x}{\rho}\right)^nn^{-\alpha}
\end{align*}
induce convergent sums. If $\alpha\in(1, 2]$ then $x^2C''(x)$ is only finite for $x\in(0,\rho)$ due to the asymptotic behavior of
\begin{align*}
x^n[x^n]x^2C''(x)\sim n^2 \lvert\mathcal{C}_n\rvert\frac{x^n}{n!}\sim b\left( \frac{x}{\rho}\right)^n n^{1-\alpha}.
\end{align*}
For $\alpha>2$ we always have that $\rho^{2}C''(\rho)<\infty$. From \eqref{prelim:eq:mean} we readily obtain that
\begin{align}
\label{proof:finitemean}
\mathbb{E}[\lvert\Gamma C(x)\rvert]<\infty,\;&x\in(0,\rho],\alpha>1, \\
\label{proof:finitevariance}
\mathrm{Var}(\lvert\Gamma C(x)\rvert)<\infty,\;& 
\begin{cases}
x\in(0,\rho), \alpha\in(1,2] \\
x\in(0,\rho], \alpha>2
\end{cases}\!\!, \\
\label{proof:infinitevariance}
\mathrm{Var}(\lvert\Gamma C(\rho)\rvert)=\infty,\;& \alpha\in(1,2].
\end{align}
By applying Lemma \ref{main:lem:uniquesolution} we infer that
\begin{align}
\label{proof:uniquex}
\frac{xC'(x)}{C(x)}=\frac{1}{\lambda}
\end{align}
has a unique solution for every $\lambda\in[\lambda^{*},1)$, where
\begin{align*}
\lambda^{*}:=\frac{C(\rho)}{\rho C'(\rho)}.
\end{align*}
For other $\lambda\in(0,\lambda^{*})$ there is no solution to \eqref{proof:uniquex}. A straightforward computation shows
\begin{align}
\label{proof:gammaequaln}
\mathbb{P}(\lvert\Gamma C(\rho)\rvert=n)\sim\frac{b}{C(\rho)}n^{-(1+\alpha)}
\end{align}
and
\begin{align}
\label{proof:gammagreatern}
\mathbb{P}(\lvert\Gamma C(\rho)\rvert> n)\sim \frac{b}{\alpha C(\rho)}n^{-\alpha}.
\end{align}
Note that Equation \eqref{prelim:eq:gnasprob} holds for any $x$ such that $G(x)$ and hence $C(x)$ is finite. We distinguish three cases, each one requiring a different choice for $x$.

$\boldsymbol{1^{st}\text{ }case}$: Assume that $\lambda\in(0,\lambda^{*})$ and $\alpha>1$. Set $x=\rho$ and note that $\mathbb{E}[\lvert\gamma(\rho)\rvert]=1/\lambda^{*}$. Due to \eqref{proof:finitemean} and \eqref{proof:gammaequaln} the first two conditions of Lemma \ref{prelim:lem:largedevi} are fulfilled. In order apply Lemma \ref{prelim:lem:largedevi} it remains to show that there is an $\varepsilon > 0$ such that $n \ge (1/\lambda^* +  \varepsilon)N$. Indeed, let $\xi$ be such that $N = \lambda n + \xi$. Then $|\xi| \in [0,1]$ and for sufficiently large $n$
\[
	\frac{n}N =  \frac{n}{\lambda n + \xi} = \frac1\lambda + o(1) > \frac1{\lambda^*},
\]
so that Lemma \ref{prelim:lem:largedevi} can be applied. We obtain
\begin{align*}
\mathbb{P}\left(\sum\limits_{i\leq N}\lvert\gamma_i(\rho)\rvert=n\right)\sim \frac{b\lambda}{C(\rho)(1-\lambda/\lambda^{*})^{1+\alpha}}n^{-\alpha}.	
\end{align*}
With Lemma \ref{prelim:lem:sumiid} we then immediately obtain the claimed value for $\lvert\mathcal{G}_{n,N}\rvert$.

$\boldsymbol{2^{nd}\text{ }case}$: Consider $\lambda=\lambda^{*}$.
Let $x=\rho$ and assume that $\alpha\in(1, 2)$. Equation \eqref{proof:gammagreatern} shows that Lemma \ref{prelim:lem:janson} is applicable to $\sum_{i\leq N}\lvert\gamma_i(\rho)\rvert$ with $a_N=N/\lambda^{*}$ and $b_N=N^{1/\alpha}$. It holds that $a_{N}+\mathcal{O}(1)=n$ and $b_{N}\sim (\lambda^{*} n)^{1/\alpha}$. Thus, noting that $\lvert\Gamma C(\rho)\rvert$ is taking values in the lattice $\mathbb{N}$ as $\mathcal{C}_1$ and $\mathcal{C}_n$ are non-empty for sufficiently large $n$, Lemma \ref{prelim:lem:generallocallimit} gives us
\begin{align*}
&\mathbb{P}\left(\sum\limits_{i\leq N}\lvert\gamma_i(\rho)\rvert=n\right) \sim \left(\frac{\alpha C(\rho)}{\lambda^{*}b\lvert\Gamma(1-\alpha)\rvert}\right)^{1/\alpha}\frac{1}{\lvert\Gamma(-1/\alpha)\rvert}n^{-1/\alpha}.
\end{align*}
This proves the case $\lambda=\lambda^{*}$ of Theorem \ref{main:bigthm:1}.

If $\alpha=2$, we use the second part of Lemma \ref{prelim:lem:janson} to obtain 
\begin{align*}
\frac{\sum_{i\leq  N}\lvert\gamma_i(\rho)\rvert-\frac{N}{\lambda^{*}}}{\sqrt{N\log N}}\overset{(\mathrm{d})}{\rightarrow}\mathcal{N}\left(0,\frac{b}{2C(\rho)}\right).
\end{align*}
Again, by Lemma \ref{prelim:lem:generallocallimit}, and using that $N\log N\sim \lambda^{*} n\log(\lambda^{*}n)$.
\begin{align*}
\mathbb{P}\left(\sum\limits_{i\leq N}\lvert\gamma_i(\rho)\rvert=n\right)\sim \sqrt{\frac{C(\rho)}{b\lambda^{*}\pi}}(n\log \lambda^{*} n)^{-1/2}, 
\end{align*}
which gives us the terms $c_2$ and $(n\log(\lambda^{*} n))^{-1/2}$ in Theorem \ref{main:bigthm:2}. 

For the case $\alpha>2$ in Theorem \ref{main:bigthm:2} the variance $\sigma^2$ of $\lvert\Gamma C(\rho)\rvert$ is finite due to \eqref{proof:finitevariance}. As $\mathcal{C}_n$ is non-empty for sufficiently large $n$, there exists $q>0$ such that $q_i=q$ in Lemma \ref{prelim:lem:classlocal} for all $i\in\mathbb{N}$ and consequently $\sup_{k\leq N}\sigma^2N/\sum_{i\leq N}q_i=\sigma^2/q<\infty$. Hence, we can use Lemma \ref{prelim:lem:classlocal} with $a_N=N/\lambda^{*}$ and $b_N=\sigma N^{1/2}$, which leads to
\begin{align*}
\mathbb{P}\left(\sum\limits_{i\leq N}\lvert\gamma_i(\rho)\rvert=n\right)\sim\frac{1}{\sqrt{2\pi\sigma^2\lambda^{*}}}n^{-1/2}.
\end{align*}

$\boldsymbol{3^{rd}\text{ }case}$: The last case $\lambda\in(\lambda^{*},1)$ is treated similarly to $\alpha>2$ in the second case. Let $x_\lambda$ solve \eqref{proof:uniquex} and consider $\alpha>1$. The restriction $\lambda\in(\lambda^{*},1)$ guarantees that $x_\lambda\in(0,\rho)$. According to Equations \eqref{proof:finitemean} and \eqref{proof:finitevariance} it holds that $\lvert\Gamma C(x_\lambda)\rvert$ has finite mean and finite variance $\sigma_\lambda^2$. Following the arguments of $\alpha>2$ in the second case, we deduce
\begin{align*}
\mathbb{P}\left(\sum\limits_{i\leq N}\lvert\gamma_i(\rho)\rvert=n\right)\sim\frac{1}{\sqrt{2\pi\sigma_\lambda^2\lambda}}n^{-1/2}.
\end{align*}
According to the latter equation, we immediately obtain $c_{+}(\lambda)$.
To compute $\sigma_\lambda^2$, simply combine $\mathbb{E}[\lvert\Gamma C(x_\lambda)\rvert]=1/\lambda$ with \eqref{prelim:eq:mean}.
\end{proof} \hfill$\square$

%% file: examples.tex
In this section we shall demonstrate that the main result is applicable to so-called block-stable graph classes. Such classes are characterized as follows. Let $\mathcal{G}$ denote a subspecies of the species of graphs, $\mathcal{C}$ the subspecies of connected graphs in $\mathcal{G}$ and $\mathcal{B}$ the subspecies containing all 2-connected graphs in $\mathcal{C}$ or only two vertices joined by an edge. The blocks of $G$ are the maximal connected sub-graphs in $G$ not containing a cutvertex (of themselves). We call $\mathcal{G}$ or $\mathcal{C}$ block-stable, if $\mathcal{B}$ is non-empty and $G\in\mathcal{G}$ if and only if every block of $G$ is in $\mathcal{B}$ or an isolated vertex.  A more detailed view on block-stable graph classes can be found in \cite[Chapter 4.2]{Leroux1998}, where it is proven in Proposition 2 that 
\begin{align}
\label{examples:blockdecomp}
xC'(x)=x\exp(B'(xC'(x)))
\end{align}
if $\mathcal{C}$ is block-stable. An immediate consequence of this equation is
\begin{align}
\label{examples:blockdecomp2}
C(x)=xC'(x)-xC'(x)B'(xC'(x))+B(xC'(x)).
\end{align}
According to \eqref{proof:uniquex} we will need to investigate the equation $xC'(x)/C(x)=1/\lambda$, which is equivalent to 
\begin{align}
\label{examples:solvelambda}
1-B'(xC'(x))+\frac{B(xC'(x))}{xC'(x)}=\lambda.
\end{align}
Indeed, the latter equation yields a simplification as the the egf $B(x)$ is known explicitly in many cases and consequently for given $\lambda$ it is possible to compute $xC'(x)$ analytically. Note that Equation \eqref{examples:solvelambda} allows for a unique solution for certain values of $\lambda$, compare to Lemma \ref{main:lem:uniquesolution}.

Let $\rho$ and $R$ denote the radii of convergence for $C(x)$ and $B(x)$ and set $\zeta:=\rho C'(\rho)$. If a block-stable class of connected graphs $\mathcal{C}$ additionally satisfies that $\zeta\in(0,R)$, then $\mathcal{C}$ is termed \textit{subcritical}. Equivalently, it is shown in \cite{panagiotou2016} (Lemma 3.8) that $\mathcal{C}$ is subcritical if and only if $RB''(R)>1$. In this case $\zeta B''(\zeta)=1$. Combining this with Equation \eqref{examples:blockdecomp}, we deduce
\begin{align}
\label{examples:eq:rho}
\rho = \zeta \exp (-B'(\zeta))
\end{align}
and according to Equation \eqref{examples:solvelambda} the value $\lambda^{*}$ in Equation \eqref{main:eq:lambdastar} at which the phase transition appears is computed as
\begin{align}
\label{examples:eq:lambdastar}
\lambda^{*}=1-B'(\zeta)+\frac{B(\zeta)}{\zeta}.
\end{align}
By \cite[Corollary 3.9]{panagiotou2016} the quantity $\lvert\mathcal{C}_n\rvert$ is given by
\begin{align}
\label{examples:eq:C_n}
\lvert\mathcal{C}_n\rvert\sim \frac{\zeta}{\sqrt{2\pi(1+\zeta^2B'''(\zeta))}}n^{-5/2}\rho^{-n}n!.
\end{align}
Hence, the assumptions of Theorem \ref{main:bigthm:1} are fulfilled for any subcritical class of block-stable graphs.
 
\begin{theorem}\label{examples:thm:subcrit}
Let $\mathcal{C}$ denote a subcritical class of block-stable graphs. Then $\alpha=3/2
$ and the constants in Theorem \ref{main:bigthm:1} are as depicted in Figure \ref{examples:fig:recipe}.
\end{theorem}

\begin{figure}[h]
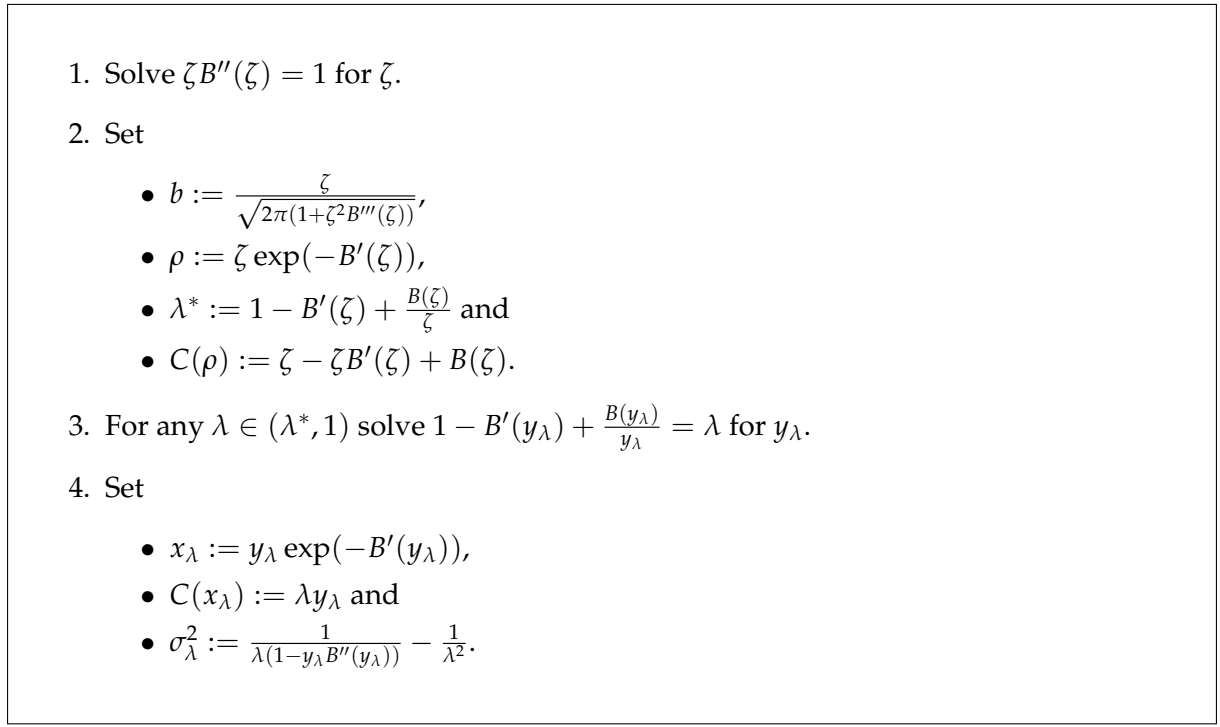
\caption{Constants of Theorem \ref{main:bigthm:1}}\label{examples:fig:recipe}\begin{framed}\begin{enumerate}
\item Solve $\zeta B''(\zeta)=1$ for $\zeta$.
\item Set 
	\begin{itemize}
	\item $b:=\frac{\zeta}{\sqrt{2\pi(1+\zeta^2B'''(\zeta))}}$,
	\item $\rho:=\zeta\exp(-B'(\zeta))$,
	\item $\lambda^{*}:=1-B'(\zeta)+\frac{B(\zeta)}{\zeta}$ and
	\item $C(\rho):=\zeta-\zeta B'(\zeta)+B(\zeta)$.
	\end{itemize}
\item For any $\lambda\in(\lambda^{*},1)$ solve $1-B'(y_\lambda)+\frac{B(y_\lambda)}{y_\lambda}=\lambda$ for $y_\lambda$.
\item Set
	\begin{itemize}
	\item $x_\lambda:=y_\lambda\exp(-B'(y_\lambda))$,
	\item $C(x_\lambda):=\lambda y_\lambda$ and
	\item $\sigma_\lambda^2:=\frac{1}{\lambda (1-y_\lambda B''(y_\lambda))}-\frac{1}{\lambda^2}$.
	\end{itemize}
\end{enumerate}\end{framed}\end{figure}

Whereas the constants in Step $2$ of Figure \ref{examples:fig:recipe} result from the fact that we set $\zeta=\rho C'(\rho)$ combined with Equations \eqref{examples:eq:rho}, \eqref{examples:eq:lambdastar} and ${C(\rho)}/{\rho C'(\rho)}={1}/{\lambda^{*}}$, Step $3$ is more involved; by differentiating \eqref{examples:blockdecomp} and performing algebraic manipulations, we first observe
\begin{align*}
(xC'(x))'=\frac{xC'(x)}{x(1-xC'(x)B''(xC'(x)))}.
\end{align*}
Further
$
{x(xC'(x))'}= {x^2C''(x)+xC'(x)}
$
and hence we obtain for $\lambda\in(\lambda^{*},1)$ in Equation \eqref{main:eq:sigma}
\begin{align*}
\sigma_\lambda^2=\frac{x_\lambda(x_\lambda C'(x_\lambda))'}{C(x_\lambda)}-\frac{1}{\lambda^2},
\end{align*}
where $x_\lambda$ denotes the unique solution to \eqref{main:eq:uniquesolution}. The switch of variables $x_\lambda C(x_\lambda)=y_\lambda$ yields the identities in Step 4 of Figure \ref{examples:fig:recipe} by applying \eqref{examples:blockdecomp}.

In the following, we apply Theorem \ref{examples:thm:subcrit} to the species of trees $\mathcal{T}$, cacti graphs $\mathcal{C}$ and Husimi trees $\mathcal{H}$. In Figure \ref{examples:fig:ch} the functions appearing in Theorem \ref{main:bigthm:1} are plotted for the respective species.

\subsection*{Forests of Trees}
It is well known that $\mathcal{T}$ is block-stable and $\mathcal{B}$ is the subspecies of edges, i.e. $B(x)=x^2/2$. Since $xB''(x)\rightarrow\infty$ as $x$ tends to infinity, we see that $\mathcal{T}$ is subcritical. Hence, we follow the computations in Figure \ref{examples:fig:recipe} to show the validity of Theorem \ref{main:thm:forests}.

\subsection*{Forests of Cacti Graphs}
Objects in $\mathcal{C}$ are connected graphs in which any two cycles have at most one vertex in common. Thus, $\mathcal{B}$ is the subspecies of cycles or pairs of vertices joined by an edge. This readily yields the egf\ of $\mathcal{B}$ 
\begin{align*}
B(x)=\frac{x^2}{4}-\frac{x}{2}-\frac{\log(1-x)}{2},
\end{align*}
cf.\ \cite{panagiotou2016}. Further, $\mathcal{C}$ is subcritical as $xB''(x)\rightarrow\infty$ as $x$ tends to $1$. We develop the constants of Theorem \ref{main:bigthm:1} step by step as in Figure \ref{examples:fig:recipe}. Let $\zeta\approx0.45631$ denote the unique solution to 
\begin{align*}
\zeta+\frac{(2-\zeta)\zeta^2}{2(1-\zeta)^2}=1.
\end{align*}
Then
\begin{itemize}
\item $b={\zeta}\,\big({2\pi(1+\zeta^2/(1-\zeta)^3)}\big)^{-1/2}\approx 0.12014$,
\item $\rho=\zeta\exp\left\{-\zeta-{\zeta^2}/{2(1-\zeta)}\right\}\approx0.23874$,
\item $\lambda^{*}={1}/{2}- {3\zeta}/{4}- {\zeta^2}/{2(1-\zeta)}-{\log(1-\zeta)}/{2\zeta}\approx0.63400$ and
\item $C(\rho)=\zeta\lambda^{*}\approx0.28930$.
\end{itemize}
For $\lambda\in(\lambda^{*},1)$ let $y_\lambda$ be the unique solution to 
\begin{align*}
\frac{1}{2}-\frac{3y_\lambda}{4}-\frac{y_\lambda^2}{2(1-y_\lambda)}-\frac{\log(1-y_\lambda)}{2y_\lambda}=\lambda.
\end{align*}
Again, we determine the constants in dependence of $y_\lambda$, namely
\begin{itemize}
\item $x_\lambda=y_\lambda\exp\left\{-y_\lambda-\frac{y_\lambda^2}{2(1-y_\lambda)}\right\}$,
\item $C(x_\lambda)=\lambda y_\lambda$ and
\item $\sigma_\lambda^2=-\frac{2(1-y_\lambda)^2}{\lambda(y_\lambda^3-4y_\lambda^2+6y_\lambda-2)}-\frac{1}{\lambda^2}$.
\end{itemize}

\begin{figure}
	\label{examples:fig:ch}
	\begin{center}
		\includegraphics[scale=0.4]{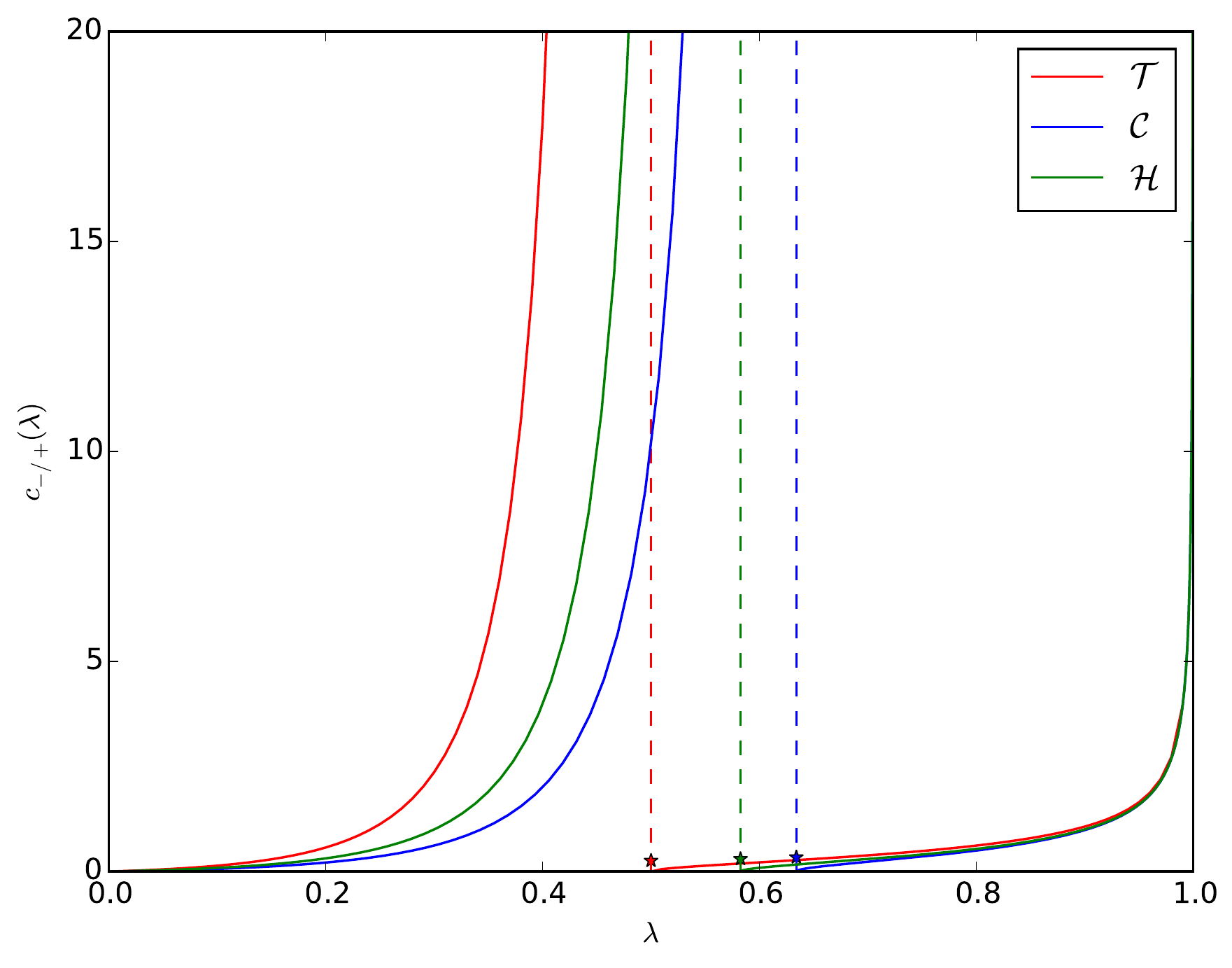}
		\caption{The functions $c_{-/+}$ (solid line) and the constant $c$ (star) from Theorem \ref{main:bigthm:1} for the species of trees $\mathcal{T}$ (red), cacti graphs $\mathcal{C}$ (green) and Husimi trees $\mathcal{H}$ (blue). The respective dotted lines represent the critical point $\lambda^{*}$.}	
	\end{center}
\end{figure}

\subsection*{Forests of Husimi Trees}
For $\mathcal{H}$ the subspecies of 2-connected graphs $\mathcal{B}$ is the species of complete graphs, where $\lvert\mathcal{B}_n\rvert=1$ for all $n\geq 2$. It may easily be verified that the egf
\begin{align*}
B(x)=\mathrm{e}^x-x-1
\end{align*}
is analytic and $x\mathrm{e}^{x}\rightarrow\infty$ as $x$ tends to infinity, i.e. $\mathcal{H}$ is subcritical. Again, we adhere to the steps given in Figure \ref{examples:fig:recipe} in order to determine the constants in Theorem \ref{main:bigthm:1}. Let $\zeta\approx0.56714$ be the solution to
$
\zeta\mathrm{e}^\zeta=1.
$ We compute
\begin{itemize}
\item $b={\zeta} \, \big(2\pi(1+\zeta^2\mathrm{e}^{\zeta})\big)^{-1/2}\approx 0.18073$,
\item $\rho=\zeta\exp\left\{-\mathrm{e}^{\zeta}+1\right\}\approx0.26438$,
\item $\lambda^{*}=1-{1}/{\zeta}+{\mathrm{e}^{\zeta}}/{\zeta}-\mathrm{e}^{\zeta}\approx0.58250$ and
\item $H(\rho)=\zeta-1+\mathrm{e}^{\zeta}-\zeta\mathrm{e}^{\zeta}\approx0.33036$.
\end{itemize}
For the constants in dependence of $y_\lambda$, as the unique solution to
\begin{align*}
1-\frac{1}{y_\lambda}+\frac{\mathrm{e}^{y_\lambda}}{y_\lambda}-\mathrm{e}^{y_\lambda}=\lambda
\end{align*}
for $\lambda\in(\lambda^{*},1)$, we find
\begin{itemize}
\item $x_\lambda=y_\lambda\exp\left\{-\mathrm{e}^{y_\lambda}+1\right\}$,
\item $H(x_\lambda)=\lambda y_\lambda$ and
\item $\sigma_\lambda^2=\frac{1}{\lambda\left(1-y_\lambda\mathrm{e}^{y_\lambda}\right)}-\frac{1}{\lambda^2}$.
\end{itemize}